\documentclass[11pt]{article}
\usepackage[utf8]{inputenc}
\usepackage[dvipsnames]{xcolor}
\usepackage{graphicx}
\usepackage{amsmath}
\usepackage{amsfonts}
\usepackage{amssymb}
\usepackage{amsthm}
\usepackage{thmtools}
\usepackage[margin=25mm]{geometry}

\usepackage{caption}
\usepackage{enumitem,calc}
\usepackage[longnamesfirst,numbers,sort&compress]{natbib}

\usepackage[unicode=true,hypertexnames=false]{hyperref}
\hypersetup{
colorlinks,
linkcolor={black},
citecolor={black},
urlcolor={blue!60!black}
}

\usepackage[noabbrev,capitalise,nameinlink]{cleveref}
\crefname{lemma}{Lemma}{Lemmas}
\crefname{theorem}{Theorem}{Theorems}
\crefname{corollary}{Corollary}{Corollaries}
\crefname{claim}{Claim}{Claims}
\crefname{proposition}{Proposition}{Propositions}
\crefname{figure}{Figure}{Figures}

\theoremstyle{plain}
\newtheorem{theorem}{Theorem}
\newtheorem{proposition}[theorem]{Proposition}
\newtheorem{lemma}[theorem]{Lemma}

\newenvironment{proofofclaim}[1][Proof.]{\proof[#1]}{\endproof}

\newtheorem{claim}{Claim}
\usepackage{etoolbox}
\AtEndEnvironment{proof}{\setcounter{claim}{0}}

\theoremstyle{definition}

\DeclareMathOperator{\dist}{dist}

\newcommand{\scr}[1]{\mathcal{#1}}
\newcommand{\ds}[1]{\mathbb{#1}}

\newcommand{\defn}[1]{\textcolor{Maroon}{\emph{#1}}}
\newcommand{\RanIn}[2]{\mathbb{#1}^{\geq#2}}

\renewcommand{\geq}{\geqslant}
\renewcommand{\leq}{\leqslant}

\title{\bf An alternative characterisation of graphs\\ 
quasi-isometric to graphs of bounded treewidth}
\author{Marc Distel\,\footnotemark[2]}
\date{\today}

\footnotetext[2]{School of Mathematics, Monash University, Melbourne, Australia (\texttt{marc.distel@monash.edu}). Research supported by Australian Government Research Training Program Scholarship.}

\begin{document}
\maketitle

\setlength{\parindent}{0pt}
\setlength{\parskip}{10pt}

\begin{abstract}
    Quasi-isometry is a measure of how similar two graphs are at `large-scale'. Nguyen, Scott, and Seymour [arXiv:2501.09839]  and Hickingbotham [arXiv:2501.10840] independently gave a characterisation of graphs quasi-isometric to graphs of treewidth $k$. In this paper, we give a new characterisation of such graphs. Specifically, we show that such graphs $G$ are characterised by the existence of a partition whose quotient has treewidth at most $k$ and such that each part has bounded weak diameter in $G$. The primary contribution of our characterisation is a structural description of graphs that admit such a quasi-isometry. This differs from the characterisation mentioned above, which primarily shows the existence of such a quasi-isometry. The characterisations are complementary, and neither immediately implies the other.
\end{abstract}

\newpage

\section{Introduction}

A tree-decomposition of a graph $G$ identifies a tree-like structure in the graph. Specifically, for a tree $T$, a \defn{$T$-decomposition} of $G$ consists of a collection of subsets $(J_t:t\in V(T))$ (called \defn{bags}) of $V(G)$ such that (1) for each $v\in V(G)$, the vertices $t\in V(T)$ such that $v\in J_t$ induce a nonempty subtree of $T$, and (2) for each $uv\in E(G)$, there exists $t\in V(T)$ such that $u,v\in J_t$. A \defn{tree-decomposition} is a $T$-decomposition for some tree $T$.

A particularly important property of a tree-decomposition is its \defn{width}, which is the maximum size of a bag minus $1$. The \defn{treewidth} of a graph $G$ is the minimum width of a tree-decomposition of $G$, and can be thought of as a measure of how similar a graph is to a tree.

Treewidth is an important and well-studied properties of graphs, but is unbounded on several relatively simple classes of graphs. For example, the complete graph $K_n$ has treewidth $n-1$. This can restrict the usefulness of the parameter, as having even one large clique within the entire structure of the graph immediately disqualifies it from having small treewidth. One way to bypass this restriction is to only require that a graph is `metrically similar' to a graph of bounded treewidth. Specifically, we ask that they are related via `quasi-isometry'.

For $c\in \RanIn{R}{1}$, a map $\phi:V(G)\mapsto V(H)$ is a \defn{$c$-quasi-isometry} from a graph $G$ to a graph $H$ if:

\begin{enumerate}
    \item for each pair of distinct $u,v\in V(G)$,
    \begin{enumerate}
        \item $\dist_G(u,v)\leq c\dist_H(\phi(u),\phi(v))+c$, and
        \item $\dist_H(\phi(u),\phi(v))\leq c\dist_G(u,v)+c$, and
    \end{enumerate}
    \item for each $h\in V(H)$, there exists $v\in V(G)$ such that $\dist_H(h,\phi(v))\leq c$.
\end{enumerate}
The first condition can be rewritten as $\dist_G(u,v)/c - 1\leq \dist_H(\phi(u),\phi(v))\leq c\dist_G(u,v)+c$. The second condition is to ensure that the map is `almost' surjective. As such, we refer to it as the \defn{quasi-surjectivity} condition. It gives rise to a $c'$-quasi-isometry from $H$ to $G$, where $c'$ is not much larger than $c$.

A graph $G$ is \defn{$c$-quasi-isometric} to a graph $H$ if there exists a $c$-quasi-isometry from $G$ to $H$. By quasi-surjectivity, $H$ is $c'$-quasi-isometric to $G$.

A natural question is whether a `coarse' equivalent of the width of a tree-decomposition gives a quasi-isometry to a graph of bounded treewidth. Specifically, instead of counting the number of vertices in a bag, we count the fewest number of sets of some fixed weak diameter needed to cover a bag. Here, the \defn{weak diameter} of a set $S\subseteq V(G)$ in a graph $G$ is the maximum distance between two vertices in $G$ (as measured in $G$, not $G[S]$).

\citet{Nguyen2025} showed the following.

\begin{theorem}
    \label{sufficientCondition}
    Let $k\in \ds{N}$ and $\ell\in \RanIn{R}{0}$. If a graph $G$ admits a tree-decomposition where each bag is the union of at most $k$ sets each of weak diameter at most $\ell$, then $G$ is $2(k+2)\ell$-quasi-isometric to a graph of treewidth at most $k$.
\end{theorem}

\citet{Hickingbotham2025} independently showed a similar result but with a weaker bound on the treewidth ($2k-1$), and a different coarseness.

The coarseness bound in \cref{sufficientCondition} is not directly stated by \citet{Nguyen2025}, but it can be deduced from their proof (specifically, using \citep[(2.3)]{Nguyen2025} and the observation in the paragraph preceding it with $r:=\ell$).

In \cref{sufficientCondition}, the bags are a union of $k$ sets, but the quasi-isometry is to a graph of treewidth $k$. So the bags contain $k+1$ vertices. This increase from $k$ to $k+1$ is annoying, but necessary for all $k\in \RanIn{N}{1}$, as noted by \citet{Nguyen2025}. However, \citet{Nguyen2025} did show that by `tweaking' the definition of treewidth slightly (obtaining a new parameter `pseudo-treewidth'), one can obtain an exact equality of the constants (`pseudo-treewidth' $k-1$).

Up to a change in constants, the reverse direction of \cref{sufficientCondition} also holds. Specifically, it is easily shown (and was shown by both \citet{Nguyen2025} and \citet{Hickingbotham2025}, albeit with slightly different bounds on the weak diameter) that if a graph is $c$-quasi-isometric to a graph of treewidth at most $k$, then it admits a tree-decomposition where each bag is the union of at most $k+1$ sets of weak diameter at most $2c^2+c$. (To get the exact bound on the diameter, follow the proof of \citep[Lemma 6]{Hickingbotham2025}, but use the precise bound $\dist_H(\phi(u),\phi(v))\leq c\dist_G(u,v)+c$ rather than the cruder bound $\dist_H(\phi(u),\phi(v))\leq c\dist_G(u,v)+c^2$.)

As such, \cref{sufficientCondition} provides a (loose) characterisation of graphs quasi-isometric to graphs of bounded treewidth. The ease at which the reverse direction is obtained means that the primary contribution of this result is the ability to construct a quasi-isometry to a graph of bounded treewidth (by finding such a tree-decomposition).

In this paper, we present a new characterisation of graphs quasi-isometric to graphs of bounded treewidth. In contrast to \cref{sufficientCondition}, the primary contribution of our characterisation is a structural description of graphs that are known to admit such a quasi-isometry. As such, our result is complementary to \cref{sufficientCondition}. Further, neither our result nor \cref{sufficientCondition} easily imply the other.

Our result is described using graph partitions. A \defn{partition} $\scr{P}$ of a graph $G$ is a collection of disjoint subsets of $V(G)$ whose union is $V(G)$. The elements of $\scr{P}$ are called \defn{parts}. For technical reasons, we allow empty parts. We also allow duplicate parts, so $\scr{P}$ is not a set. However, we remark that the only duplicate parts that can exist are empty parts. We say that a partition is \defn{proper} if each part is nonempty. In this case, $\scr{P}$ is a set.

The \defn{quotient} of $G$ by $\scr{P}$, denoted \defn{$G/\scr{P}$}, is the simple graph with vertex set $\scr{P}$ (allowing for multiple vertices to have the duplicate label $\emptyset$), where distinct parts $P_1,P_2\in \scr{P}$ are adjacent if and only if there exists $uv\in E(G)$ with $u\in P_1$ and $v\in P_2$. $\scr{P}$ is \defn{connected} if each $P\in \scr{P}$ induces a connected subgraph of $G$. If $\scr{P}$ is connected, then $G/\scr{P}$ is a minor of $G$. However, if $\scr{P}$ is not connected, then $G/\scr{P}$ need not be a minor of $G$.

It is easy to show that for any $d\in \RanIn{R}{0}$ and any proper partition $\scr{P}$ of a graph $G$ into parts that have weak diameter in $G$ at most $d$, $G$ is $(d+1)$-quasi-isometric to $G/\scr{P}$. (The `proper' condition is needed for quasi-surjectivity.) Indeed, this was already observed by \citet{Hickingbotham2025}, albeit without proof. (For a proof see, for example, \citet[Lemma ~3.7]{Distel2025Cw} or \citet[Lemma~2.2]{Albrechtsen2025K2t}.) However, it is not obvious whether, given a quasi-isometry to a graph $H$, one can construct such a partition whose quotient is (isomorphic to) $H$, or even a subgraph or minor of $H$.

If we relax our goal to instead say that we want the quotient to have the same `structural properties' as $H$ (that is, belong to some structurally defined class that contains $H$), then we can give an affirmative answer for bounded treewidth.

\begin{theorem}
    \label{twMain}
    Let $c\in \RanIn{R}{1}$ and $k\in \ds{N}$, and let $G$ be a graph 
    that is $c$-quasi-isometric to a graph $H$ of treewidth at most $k$. Then $G$ admits a proper partition $\scr{P}$ whose quotient has treewidth at most $k$ and whose parts have weak-diameter in $G$ at most $4(k+1)c^2+c$.
\end{theorem}

The partition $\scr{P}$ is not necessarily connected, so $G/\scr{P}$ may not be a minor of $G$.

Since a proper partition into bounded weak diameter parts gives rise to a quasi-isometry, \cref{pwMain} gives an characterisation of graphs quasi-isometric to bounded treewidth graphs. Unlike \cref{sufficientCondition}, our characterisation gives the same treewidth, although the coarseness is still loose.

The tree indexing the tree-decomposition of $G/\scr{P}$ is the same tree that indexes the tree-decomposition of $H$. So we obtain an analogous result for graphs quasi-isometric to graphs of bounded pathwidth. (A \defn{path-decomposition} is a $P$-decomposition for some path $P$, and the \defn{pathwidth} is the minimum width of a path-decomposition.)

\begin{theorem}
    \label{pwMain}
    Let $c\in \RanIn{R}{1}$ and $k\in \ds{N}$, and let $G$ be a graph 
    that is $c$-quasi-isometric to a graph $H$ of pathwidth at most $k$. Then $G$ admits a proper partition $\scr{P}$ whose quotient has pathwidth at most $k$ and whose parts have weak-diameter in $G$ at most $4(k+1)c^2+c$.
\end{theorem}

We remark that \citet{Nguyen2025} showed an analogous version of \cref{sufficientCondition} for pathwidth (replacing `tree-decomposition' with `path-decomposition' and `treewidth' with `pathwidth'). As such, we have the same pair of complementary characterisations for graphs quasi-isometric to graphs of bounded pathwidth.

\cref{twMain,pwMain} are simple consequences of a new structural result about graphs of bounded treewidth/pathwidth (\cref{compressingFuncExists}), which may be of independent interest. \cref{SecCompressing} explains the idea and motivation behind this new result.

\section{Preliminaries}

Let \defn{$\ds{N}=\{0,1,\dots\}$}. For $c\in \ds{R}$, we define \defn{$\RanIn{N}{c}$}$:=\ds{N}\cap [c,\infty)$ and \defn{$\RanIn{R}{c}$}$:=\ds{R}\cap [c,\infty)$. 

For $\ell\in \ds{R}$, the \defn{$\ell$-neighbourhood} of $S\subseteq V(G)$ in a graph $G$ is the set of vertices at distance at most $\ell$ from $S$ (so some $s\in S$). Note that $N_G^{\ell}(S)$ is empty if $\ell<0$, and that $S\subseteq N_G^{\ell}(S)$ if $\ell\geq 0$. If $S=\{v\}$ for some $v\in V(G)$, we instead write \defn{$N_G^{\ell}(v)$}. The \defn{$\ell$-th} power of $G$, denoted \defn{$G^{\ell}$}, is the graph with vertex set $V(G)$ and edges between distinct $u,v\in V(G)=V(G^{\ell})$ whenever $\dist_G(u,v)\leq \ell$.

For a graph $G$, a set $S\subseteq V(G)$ is a \defn{separator} between $A,B\subseteq V(G)$ if every path in $G$ from $A$ to $B$ intersects $S$. In this case, we say that $S$ \defn{separates} $A$ from $B$ in $G$. Note that for each $a\in A$ and $b\in B$, there exists $s\in S$ such that $\dist_G(a,b)=\dist_G(a,s)+\dist_G(b,s)$.

A \defn{class} of graphs is a collection of graphs closed under isomorphism. A class is \defn{hereditary} if any induced subgraph of a graph in the class is also in the class.

Given graphs $G,H$, we say that an \defn{$H$-indexed partition} of $G$ is a pair $(\scr{P},\psi)$, where $\scr{P}$ is a partition of $G$ and $\psi$ is a bijection from $V(H)$ to $\scr{P}$. We usually leave this bijection implicit, and write the partition $\scr{P}$ as $(P_h:h\in V(H))$ (where $P_h:=\psi(h)$ for each $h\in V(H)$). We call $P_x$, $x\in V(H)$, the \defn{parts} of $(P_h:h\in V(H))$, and we say that $(P_h:h\in V(H))$ is \defn{proper} if $\scr{P}$ is proper (equivalently, each part is nonempty). Say that $(P_h:h\in V(H))$ is an \defn{$H$-partition} if distinct $h,h'\in V(H)$ are adjacent in $H$ whenever there exists $uv\in E(G)$ with $u\in P_h$ and $v\in P_{h'}$, then $h$ and $h'$ are adjacent in $H$. In this case, if $(P_h:h\in V(H))$ is proper, then $\psi$ is an isomorphism between a spanning subgraph of $H$ and $G/\scr{P}$.

We think of a partition $\scr{P}$ of a graph $G$ as being $G/\scr{P}$-indexed, using the identity as the bijection. So $\scr{P}$ is a $G/\scr{P}$-partition.

\section{Compressing Partitions}
\label{SecCompressing}

The launching point for our main result is the following proposition, which we do not use directly, but include for intuition.

\begin{proposition}
    \label{QIHPower}
    Let $c\in \RanIn{R}{1}$, and let $G,H$ be graphs such that $G$ is $c$-quasi-isometric to $H$. Then $G$ admits an $H^{2c}$-partition such that each part has weak diameter in $G$ at most $c$.
\end{proposition}

\begin{proof}
    Let $\phi$ be a $c$-quasi-isometry from $G$ to $H$. For each $h\in V(H)$, let $P_h:=\{\phi^{-1}(h)\}$. Observe that $(P_h:h\in V(H))=(P_h:h\in V(H^{2c}))$ is an $H^{2c}$-indexed partition of $G$.

    Let distinct $h,h'\in V(H)=V(H^{2c})$ be such that there exists $uv\in E(G)$ with $u\in P_h$ and $v\in P_{h'}$. So $h=\phi(u)$ and $h'=\phi(v)$. Observe that $\dist_H(h,h')\leq c\dist_G(u,v)+c=c(1)+c=2c$. Thus, $h$ is adjacent to $h'$ in $H^{2c}$. Hence, $(P_h:h\in V(H^{2c}))$ is an $H^{2c}$-partition of $G$.

    Fix $h\in V(H)$, and consider any pair $u,v\in P_h$. So $\phi(u)=\phi(v)=h$, and thus $\dist_H(\phi(u),\phi(v))=0$. Hence, we have $\dist_G(u,v)\leq c\dist_H(\phi(u),\phi(v))+c=c(0)+c=c$. So $P_h$ has weak diameter in $G$ at most $c$. This completes the proof.
\end{proof}

If graph powers preserved treewidth and pathwidth, then we could apply \cref{QIHPower} and delete vertices corresponding to empty parts to immediately obtain \cref{twMain} and \cref{pwMain}. However, if the degree is unbounded, then this is false. Every star (tree with at most one non-leaf) has treewidth and pathwidth at most $1$, but the second power is a complete graph, whose treewidth grows proportionally to the number of vertices in the clique (and thus the original star). So we need to eliminate the power.

For this, we draw inspiration from the concept of `blocking' partitions, introduced by \citet{Distel2024Powers}. These are connected partitions in which every `long' (at least some fixed constant) path intersects some part at least twice. Such a partition trivially exists (by having one part for each connected component), so they aimed to construct such a partition with bounded-size parts. Their main focus was planar graphs, but more relevantly to this paper, they also briefly considered graphs of bounded treewidth.

A key feature of `blocking' parts is that each long path is compressed into a strictly shorter path in the quotient. More precisely, for a long path $v_0,\dots,v_n$, if $P_0,\dots,P_n$ is the sequence of parts containing $v_0,\dots,v_n$ respectively, then by assumption, $P_i=P_j$ for some pair of distinct $i,j\in \{0,\dots,n\}$. Thus, within $P_0,\dots,P_n$ there is a path from $P_0$ to $P_n$ in $G/\scr{P}$ of length at most $n-1$. Therefore, one can describe the structure of $G^n$ using the structure of $(G/\scr{P})^{n-1}$.

Because the partition is connected, the quotient is also planar/has bounded treewidth, so the argument can be repeated. This finds a connected partition of the quotient of bounded size, which can be thought of as a connected partition of $G$ of bounded size, with which the power can be reduced again. This can repeatedly inductively, until the power $\ell$ reaches the point where a path of length $\ell$ is not considered long.

The final result is a connected partition $\scr{P}$ of bounded size such that $G^n$ can be described using $(G/\scr{P})^{\ell}$, where $\ell<n$ is the fixed length at which a path is no longer long. $G/\scr{P}$ is still planar/has bounded treewidth, so we can summarise this by saying `large powers of planar/bounded treewidth graphs can be described using only small powers of planar/bounded treewidth graphs'.

This is ideal for our purposes, since we want to reduce $H^{2c}$ to a graph $H'$ of bounded treewidth. However, there are a few obstacles. Firstly, these partitions only exist when the degree is bounded. We fix this by relaxing the requirement that the parts have bounded size, instead asking that the parts have bounded weak diameter in $G$. This is natural, since when the degree is bounded, the parts have bounded size.

However, more critically, we cannot reduce beyond the `minimum length' when the path is no longer considered long. For bounded treewidth graphs, the partition obtained by \citet{Distel2024Powers} treats every path of length at least $3$ as long, and hence the power can be reduced to $2$. While this is a huge improvement, it still destroys the bound on the treewidth, as seen in our example with stars. We need to remove the power completely.

To achieve this, we modify the concept of blocking partitions. For us, the key property of `blocking' partitions is that distances in the graph are shrunk in the quotient. We take this as a defining property, and then tweak it by instead asking that the distances decrease in some graph $H$ for which $\scr{P}$ is an $H$-indexed partition. In fact, instead of just decreasing the power by $1$ each step, we can directly demand that two vertices that are `close' (adjacent in the power) belong to adjacent parts. This leads to the following definition.

For a graph $H$ and $\ell\in \RanIn{R}{0}$, say that an $H$-indexed partition $(P_h:h\in V(H))$ of a graph $G$ is \defn{$\ell$-compressing} (in $G$) if for each pair of distinct $h,h'\in V(G)$ with $\dist_G(P_h,P_{h'})\leq \ell$, $h$ is adjacent to $h'$ in $H$. Equivalently, $(P_h:h\in V(H))$ is a $H$-partition of $G^{\ell}$.

Observe that if $\ell\geq 1$, then any $\ell$-compressing partition of a graph is also an $H$-partition.

Note that we have dropped the requirement that the partition is connected. Because we allow $H$ to have more edges than the quotient, $H$ may not a minor of $G$ even if the quotient is. Thus, it is less important that the quotient is a minor of $G$. However, more importantly, we cannot both ask that $H=G/\scr{P}$ (with the identity map from $H$ to $\scr{P}$) and that $\scr{P}$ is connected, as such a partition cannot be $2$-compressing for large binary trees.

\begin{proposition}
    \label{connectedCounterexample}
    For each $d\in \ds{N}$, the binary tree $T$ of height $d+1$ does not admit a connected partition $\scr{P}$ that is $2$-compressing in $G$ (as a $(G/\scr{P})$-indexed partition with the identity map) and such that each $P\in \scr{P}$ has weak diameter at most $d$.
\end{proposition}

\begin{proof}
    Presume otherwise. Let $r$ be the root of $T$. So each leaf is at distance exactly $d+1$ from $r$. Let $P$ be the part containing $r$. Since $\scr{P}$ is connected, $G[P]$ is a subtree of $T$. Either $P=\{r\}$, in which case set $t:=r$, or there exists a leaf $t$ of $P$ other than $r$. Since $P$ has weak diameter at most $d$ in $G$ and since $r\in P$, $t$ is not a leaf of $T$. Thus, $t$ has two children $c_1,c_2$ in $T$, neither of which are in $P$ (by choice of $t$, and since $G[P]$ is connected and contains $r$). Since $\scr{P}$ is connected and since $t\in P$, $c_1,c_2$ are in distinct parts of $\scr{P}$ that are non-adjacent in $G/\scr{P}$. However, $\dist_G(c_1,c_2)=2$, which contradicts the fact that $\scr{P}$ is $2$-compressing in $G$.
\end{proof}

We remark (without proof) that \cref{connectedCounterexample} can be extended to other trees and graphs of bounded treewidth.

\cref{connectedCounterexample} eliminates the main advantage of connected partitions, as we cannot take $H:=G/\scr{P}$ to conclude $H\leq G$. So we find it more natural to drop the connected requirement and make no attempt to relate $G$ and $H$. Instead, we construct a tree-decomposition of $H$ directly. This allows for more flexibility in choice of $H$ and $\scr{P}$, which is important for our proof.

Because we no longer relate the graphs $G$ and $H$, it is natural to tie the existence of compressing partitions to a class of graphs. This leads to the following definition.

Let $\scr{H}$ be a class of graphs. Say that a function $f:\RanIn{R}{0}\mapsto \RanIn{R}{0}$ is a \defn{compressing function} for $\scr{H}$ if for each $\ell\in \RanIn{R}{0}$ and $H\in \scr{H}$, there exists $A\in \scr{H}$ and an $A$-indexed partition of $H$ that is $\ell$-compressing and such that each part has weak diameter in $H$ at most $f(\ell)$.

If a hereditary class $\scr{H}$ admits a compressing function, then being quasi-isometric to some $H\in \scr{H}$ is equivalent to admitting, for some $H'\in \scr{H}$, a proper $H'$-partition into bounded weak diameter parts. (See the introduction for the reverse direction.)

\begin{lemma}
    \label{compFuncToPart}
    Let $\scr{H}$ be a hereditary class of graphs, and let $f:\RanIn{R}{0}\mapsto \RanIn{R}{0}$ be a compressing function for $\scr{H}$. Then for every $c\in \RanIn{R}{1}$ and graph $G$ that is $c$-quasi-isometric to some $H\in \scr{H}$, there exists $A\in \scr{H}$ and a proper $A$-partition of $G$ such that each part has weak diameter at most $cf(2c)+c$ in $G$.
\end{lemma}

\begin{proof}
    Let $\phi$ be a $c$-quasi-isometry from $G$ to $H$. Since $f$ is a compressing function for $\scr{H}$, there exists a graph $A'\in \scr{H}$ and an $A'$-indexed partition $(P_a:a\in V(A'))$ of $H$ that is $2c$-compressing and such that each part has weak diameter at most $f(2c)$ in $H$.

    For each $a\in V(A')$, let $P_a':=\phi^{-1}(P_a)$. Let $A$ be the subgraph of $A'$ induced by the vertices $a\in V(A)$ such that $P_a'\neq \emptyset$. Observe that $(P_a':a\in V(A))$ is a proper $A$-indexed partition of $G$. Since $\scr{H}$ is hereditary and $A'\in \scr{H}$, we have $A\in \scr{H}$.

    Let $a,a'\in V(A)$ be distinct such that there exists $uv\in E(G)$ with $u\in P_a'$ and $v\in P_{a'}'$. So $\phi(u)\in P_a$ and $\phi(v)\in P_{a'}$. Observe that $\dist_H(\phi(u),\phi(v))\leq c\dist_G(u,v)+c=c(1)+c=2c$. Thus, $\dist_H(P_a,P_{a'})\leq 2c$. Since $(P_a:a\in V(A'))$ is a $2c$-compressing partition of $H$, we have that $a,a'$ are adjacent in $A'$ and thus $A$. It follows that $(P_a':a\in V(A))$ is an $A$-partition of $G$.

    Fix $a\in V(A)$, and consider any pair $u,v\in P_a'$. So $\phi(u),\phi(v)\in P_a$. As $P_a$ has weak diameter at most $f(2c)$ in $H$, we have $\dist_G(u,v)\leq c\dist_H(\phi(u),\phi(v))+c\leq cf(2c)+c$. So $P_a'$ has weak diameter at most $cf(2c)+c$ in $G$. This completes the proof.
\end{proof}

The main technical result of this paper is the following.

\begin{theorem}
    \label{compressingFuncExists}
    For each $k\in \ds{N}$, the function $\ell\mapsto 2(k+1)\ell$ is a compressing function for the class of graphs of treewidth at most $k$ and the class of graphs of pathwidth at most $k$.
\end{theorem}

\cref{twMain,pwMain} are obtained by applying \cref{compressingFuncExists,compFuncToPart} (noting that treewidth does not increase under taking subgraphs).

The remainder of the paper is devoted to proving \cref{compressingFuncExists}.

\section{Proof Explanation}

We now explain the idea behind the proof of \cref{compressingFuncExists}.

Fix $\ell\in \RanIn{R}{0}$, a tree-decomposition $(J_t:t\in V(T))$ of width $k$ of a graph $G$, and a root $r\in V(T)$ for $T$. We want to find a graph $H$, a $T$-decomposition of width $k$ of $H$, and an $\ell$-compressing $H$-indexed partition of $G$ whose parts have weak diameter at most $2(k+1)\ell$ in $G$. If this can always be done, then \cref{compressingFuncExists} quickly follows.

The key idea is relatively simple. Each part is contained within the $(k+1)\ell$-neighbourhood of some vertex, which we will call the \defn{centre} of the part. Critically, the centre of the part may not be in the part. The parts are indexed by these centres, so $V(H)$ is precisely the set of centres. Inductively working along $T$, the proof carefully picks new vertices to act as centres.

For each bag $J_t$ of the tree-decomposition, there is a corresponding set of centres $B_t$ of size at most $k+1$. The $T$-decomposition of $H$ will be $(B_t:t\in V(T))$, and distinct centres are adjacent in $H$ if and only if they are in a common `bag' $B_t$.

The `initial' set of centres $B_r$ is $J_r$. Then for $t\in V(T)\setminus \{r\}$, we inductively select the centres $B_t$ based on the centres $B_{t'}$ for the parent $t'$.

To each centre $x$, we associate not just a part, but also two special `zones', which we call the `security' and `coverage'. The coverage is the $i\ell$-neighbourhood of $x$, for a specially picked $i\in \{1,\dots,k+1\}$. These are the vertices that `could' go in the part $P_x$. However, the coverages are not disjoint, so some vertices in the coverage will be assigned to other parts.

For each $v\in V(G)$, let $r(v)$ be the (unique) vertex of $V(T)$ closest to $r$ such that $v\in J_{r(v)}$. We can create a `priority system' (total ordering) for $V(G)$, so that $v\in V(G)$ has `higher priority' than $u\in V(G)$ if $r(v)$ is a strict ancestor of $r(u)$. We assign each $v\in V(G)$ to the part $P_x$ of the highest priority centre $x$ whose coverage contains $v$ and such that $r(x)$ is an ancestor of $r(v)$ (such a centre will exist). From an inductive point of view, this centre is also the `oldest' centre, in the sense that the algorithm created it before (or at the same time as) any other alternatives. So we `prioritise' putting vertices in the `oldest' parts. Critically, this can (and frequently will) result in $x$ not being assigned to the part $P_x$.

The `security' is the $(i-1)\ell$-neighbourhood of $x$ (where $i$ is the same as for the coverage). So any vertex $v$ in the $\ell$-neighbourhood of the security is in the coverage, and could be assigned to $P_x$. In particular, if $x$ is sufficiently old, $v$ will be assigned to $P_x$.

Our goal is to ensure that whenever we `discard' a centre $x\in B_{t}\setminus B_{t'}$, every vertex $v$ in the coverage is either (a) contained in the coverage of an `older' (higher priority) centre, or (b) is in the security of a vertex in $B_t\cap B_{t'}$. In the former case, $v$ is not in the part $P_x$ (and is thus not relevant), and in the latter case, we find that every vertex in the $\ell$-neighbourhood of $x$ is in the coverage of a vertex in $B_t\cap B_{t'}$, and thus gets assigned to one of these parts by our priority system.

However, there is an obvious problem: what if $v$ is at the limit of the coverages of two (or more) centres, one of which we must discard? In this case, we necessarily have to grow the other coverage so that $v$ becomes `secure' (is contained in the security). If repeated, this can lead to the coverage (and thus the part) having unbounded weak diameter. This issue can be addressed by limiting the diameters of the security/coverage based on the `age' of the centre.

Since each bag $J_t$ is a separator, we need only consider the $k+1$ `closest' centres ($B_t$). Starting with the $\ell$-neighbourhood for the coverage (and $0$-neighbourhood for the security) for the `newest' (lowest priority) centre, for each increase in age (priority), we add an extra $\ell$-neighbourhood to the security and coverage. As we consider at most $k+1$ centres at a time, at most $k$ increases are performed. Thus, the coverage is at most the $(k+1)\ell$-neighbourhood.

We stress that we are now considering securities/coverages to be relative to a vertex $t$ of the tree. A centre will start with only a small coverage and security, but as it gets `older' (we move along the tree to a descendant), it is allowed to grow. This growth need not be consistent either. On one child, the security and coverage may grow, on another they may stay the same, and on yet another the centre may be discarded. However, critically, the security and coverage does not shrink (provided the centre is kept). Once it is grown, that growth is permanent, for as long as that centre is kept.

We also assign each $v\in J_t$ to a centre $x\in B_{t'}$, which we say $v$ `belongs to'. This is different to which part $P_{x'}$ $v$ is assigned to, as we might have $x'\notin B_{t'}$ (either because $x'$ was discarded in a previous step, or because $v$ is too far from any of the existing centres and needs a new centre). To chose $x$, instead of just asking `what is the closest centre' or `what is the oldest centre', we instead ask `what is the closest centre, relative to its age'. Specifically, we minimise the quantity $\dist_G(x,v)+i\ell$, where $i$ is the `penalty' of the centre, starting at $0$ for the oldest (highest priority) centre and increasing by $1$ with each increase in age. In the event of a tie, we break the tie using the priority.

Note that the `penalties' are the reverse order of the relative `priorities'. So for each decrease by $1$ in the penalty, we add an extra $\ell$-neighbourhood to the security and coverage.

The intuition behind the quantity $\dist_G(x,v)+i\ell$ only applies when $v$ is in the coverage of at least one centre in $B_{t'}$ (the other case is less important, as then $v$ is `far enough' away from any existing centres). In this case, we want to pick the centre $x$ whose coverage extends as `far past' $v$ as possible. That is, we maximise the distance to a vertex not in the coverage. This creates a `buffer' of vertices in the coverage that is as `wide' as possible. The centre that maximises the `width' of this buffer is precisely the centre that minimises $\dist_G(x,v)+i\ell$.

When we discard a centre, we decrease the penalties of all centres older than it by $1$, and thus add an extra $\ell$-neighbourhood to their securities and coverages. This means that in our `tied' example before, the vertex $v$ that was at the edge of two coverages is now in the security of the newer centre after the older centre $x$ is discarded. (The newer centre being discarded is not a concern, as we would never put $v$ in $P_{x'}$ over $P_x$.)

Occasionally, we must pick new centres. We do this if the bag contains `unsecure' (not in a security) vertices. These vertices become new centres. From this, we can inductively ensure that each vertex in $J_{t'}$ is secure, so the new centres (the unsecure vertices) come from $J_t\setminus J_{t'}$. Thus, we avoid discarding and then `re-adding' centres, ensuring that $(B_t:t\in V(T))$ is a tree-decomposition of $H$.

The vertices $J_t\cap J_{t'}$ are secure by induction, and `belong' to a centre using the `penalty-adjusted' system above. Each of these centres will be `kept' (be in $B_t\cap B_{t'}$). This means that we need to keep at most $|J_t\cap J_{t'}|$ centres. Since the vertices in $J_t\cap J_{t'}$ are secure, none of them become new centres. So, at worst, we need to add at most $|J_t\setminus J_{t'}|$ centres. Thus, we need at most $|J_t|$ centres for the next step ($B_t$). For technical reasons, we prefer to avoid decreasing the number of centres. So we add extra centres from $B_{t'}$ to $B_t$ until it has at least $|B_{t'}|$ but at most $\max(|J_t|,|B_{t'}|)$ centres. By induction, $B_t$ has at most $k+1$ centres.

There is one remaining concern: a vertex $u$ that is `close' to (within distance $\ell$ of) a vertex $v$ in the coverage of an old centre $x$ that we discard ($B_{t'}\setminus B_t$). In such a case, could we have $u\in P_{x'}$ for a `new' centre $x'$ (in $B_t\setminus B_{t'}$), while also having $v\in P_x$ (which is possible, as $v$ is in the coverage of $x$). Since $x$ and $x'$ will not be adjacent in $H$ (as they are not in a common bag $B_{t^*}$), this would be a problem. However, we show that this does not occur. For this, we exploit the penalty system and the fact $J_t\cap J_{t'}$ is a separator between $J_t$ and $B_{t'}$, as each centre in $B_{t'}$ first appears in a bag $J_{t^*}$ `before' $t'$ ($t^*$ is an ancestor of $t'$).

Since $J_t\cap J_{t'}$ is a separator, there is a vertex $z\in J_t\cap J_{t'}$ on the shortest path from $x$ to $v$. Note that $z$ is also in the coverage of $x$. $z$ `belongs' to a centre $y$ we keep ($y\in B_t\cap B_{t'}$). By choice of how we pick $y$, the `buffer zone' in the coverage of $y$ around $z$ is `wider' than the buffer zone from $x$. This implies that $v$ is also in the coverage of $y$.

If this centre $y$ has higher priority than $x$, then we would never put $v$ in $P_x$. If $y$ has lower priority (and thus higher penalty), then in the next step, as we discarded $x$, the penalty of $y$ will be decreased by at least $1$. Hence, $v$ will be in the security of $y$. Thus, the $\ell$-neighbourhood of $v$, including $u$, is contained in the coverage of $y$.

Our priority system will always prioritise the kept centres over the new centres (by an easy inductive argument). Thus, we would never put $u$ in a part belonging to a new centre (or any centre created thereafter). This means that the concerning scenario above cannot occur. From this, we derive that if $u,v\in V(G)$ are at distance at most $\ell$, then the centres $x,x'\in V(H)$ with $u\in P_x$ and $v\in P_{x'}$ satisfy $u,v\in B_t$ for some $t\in V(T)$ (in particular, $t\in \{r(x),r(x')\}$). Then by definition of $H$, either $u=v$ or $uv\in E(H)$. So the partition is compressing.

The weak diameter of the parts is derived from the fact that the coverage is at most the $(k+1)$-neighbourhood. Recall that we only add a vertex $x$ as a centre in $B_t\setminus B_{t'}$ if it was unsecure, and that once a vertex is secure, it stays secure. Therefore, we never discard a centre and re-add it later. It follows that $(B_t:t\in V(T))$ is a tree-decomposition of $H$.

\section{Proof}

Say that a \defn{graph-priority pair} is a pair $Q=(G,\preceq)$ where $G$ is a graph and $\preceq$ is a total ordering of $V(G)$. For distinct $u,v\in V(G)$, if $u\prec v$, then we say that $u$ has \defn{lower priority} than $v$ in $Q$, and if $v\prec u$, then we say that $u$ has \defn{higher priority} than $v$ in $Q$. We consider $u$ to have neither higher nor lower priority than itself. For a binary condition $\beta$ on $V(G)$ such that at least one vertex satisfies $\beta$, there exists exactly one vertex $v\in V(G)$ that satisfies $\beta$ and has higher priority in $Q$ than each $u\in V(G)\setminus \{v\}$ that also satisfies $\beta$. We call $v$ the vertex of \defn{maximum priority} in $Q$ that satisfies $\beta$.

Let $Q=(G,\preceq)$ be a graph-priority pair, and let $B\subseteq V(G)$. The \defn{penalty map} for $B$ in $Q$, denoted \defn{$\sigma_Q^B$}, is the map from $B$ to $\{0,\dots,|B|-1\}$ that sends each $x\in B$ to $|\{x'\in B:x\prec x'\}|$. So $\sigma_Q^B$ counts the number of vertices in $B$ of higher priority than $x$. Note that $\sigma_Q^B$ is a bijection, and that $\sigma_Q^B(x')<\sigma_Q^B(x)$ whenever $x,x'\in B$ is such that $x'$ has higher priority in $Q$ than $x$. For each $x\in B$, we call $\sigma_Q^B$ the \defn{penalty} of $x$ in $Q$.

For $\ell\in \RanIn{R}{0}$, and $B'\subseteq B$, the \defn{$(B,\ell)$-security} of $B'$ in $Q$, denoted \defn{$S_Q^{B,\ell}(B')$}, is $$\bigcup_{x\in B'}N_G^{(|B|-1-\sigma^B_Q(x))\ell}(x).$$ The \defn{$(B,\ell)$-coverage} of $B'$ in $Q$, denoted \defn{$C_Q^{B,\ell}(B')$}, is $$\bigcup_{x\in B'}N_G^{(|B|-\sigma^B_Q(x))\ell}(x).$$ Note that $B'\subseteq S_Q^{B,\ell}(B')$ (as $\sigma^B_Q(x)\leq |B|-1$), and that $N_G^{\ell}(S_Q^{B,\ell}(B'))\subseteq C_Q^{B,\ell}(B')$. If $B'$ is a singleton $\{x\}$, we instead write \defn{$S_Q^{B,\ell}(x)$} and \defn{$C_Q^{B,\ell}(x)$} for the $(B,\ell)$-security and $(B,\ell)$-coverage respectively.

For $J\subseteq V(G)$ and $x\in B$, the \defn{$(B,\ell)$-leak} of $x$ through $J$ in $Q$, denoted \defn{$L_Q^{B,\ell}(x,J)$}, is the set of vertices $v\in V(G)$ such that: 
\begin{enumerate}
    \item there exists $z\in J$ such that $\dist_G(v,z)\leq (|B|-\sigma^B_Q(x))\ell-\dist_G(x,z)$, and
    \item there does not exist $x'\in B$ with higher priority than $x$ in $Q$ such that $v\in C_Q^{B,\ell}(x')$.
\end{enumerate}

Before we start the main proof, we first need the following lemma, which says how to pick `new centres'.

\begin{lemma}
    \label{changeProviders}
    Let $Q=(G,\preceq)$ be a graph-priority pair, let $\ell\in \RanIn{R}{0}$, and let $B,J\subseteq V(G)$ be such that:
    \begin{enumerate}
        \item $J\cap S_Q^{B,\ell}(B)$ separates $B$ from $J$ in $G$, and,
        \item for each $u\in B$ and $v\in J\setminus S_Q^{B,\ell}(B)$, $u$ has higher priority than $v$ in $Q$.
    \end{enumerate}
    Then there exists $B'\subseteq V(G)$ such that:
    \begin{enumerate}
        \item $|B|\leq |B'|\leq \max(|B|,|J|)$,
        \item $J\subseteq S_Q^{B',\ell}(B')$,
        \item $B'\setminus B=J\setminus S_Q^{B,\ell}(B)$
        \item for each $x\in B\setminus B'$, $L_Q^{B,\ell}(x,J)\subseteq S_Q^{B',\ell}(B\cap B')$.
    \end{enumerate}
\end{lemma}

\begin{proof}
    For brevity, we define $\sigma:=\sigma_Q^B$, and once $B'$ is defined, we set $\sigma':=\sigma_Q^{B'}$. For $B^*\subseteq B$, we also define $(S(B^*),C(B^*)):=(S_Q^{B,\ell}(B^*),C_Q^{B,\ell}(B^*))$. If $B^*=\{x\}$ for some $x\in B$, we abbreviate $S(\{x\})$ as $S(x)$ and $C(\{x\})$ as $C(x)$. Once $B'$ is defined, for $B^*\subseteq B'$, we similarly define $(S'(B^*),C'(B^*)):=(S_Q^{B',\ell}(B^*),C_Q^{B',\ell}(B^*))$. Again, if $B^*=\{x\}$ for some $x\in B'$, we abbreviate $S'(\{x\})$ as $S'(x)$ and $C'(\{x\})$ as $C'(x)$. Finally, for $x\in B$, we define $L(x):=L_Q^{B,\ell}(x,J)$.

    Under these abbreviations, we have:
    \begin{enumerate}
        \item $S(B)\cap J$ separates $B$ from $J$ in $G$, and
        \item for each $u\in B$ and $v\in J\setminus S(B)$, $u$ has higher priority than $v$ in $Q$,
    \end{enumerate}
    and we must show that:
    \begin{enumerate}
        \item $|B|\leq |B'|\leq \max(|B|,|J|)$,
        \item $J\subseteq S'(B')$,
        \item $B'\setminus B=J\setminus S(B)$, and
        \item for each $x\in B\setminus B'$, $L(x)\subseteq S'(B\cap B')$.
    \end{enumerate}

    If $B=\emptyset$, the lemma is trivially satisfied with $B':=J$, as $J\subseteq S_Q^{J,\ell}(J)=S'(B')$ and $S(B)=\emptyset$ in this case. So we may assume $B\neq \emptyset$.

    For each $v\in J$, let $a(v)$ be the $x\in B$ that minimises $\dist_G(x,v)+\sigma(x)\ell$, and, subject to this, is of maximum priority in $Q$. Note that $a(v)$ exists as $B\neq \emptyset$.

    We will say that $v\in J$ `belongs' to the `centre' $a(v)$. We first show that the choice of `belonging' is `consistent' if a vertex $u\in J$ lies on the shortest path from $v$ to $a(v)$.

    \begin{claim}
        \label{claimAncSeps}
        If $u,v\in J$ are such that $\dist_G(v,a(v))=\dist_G(u,v)+\dist_G(u,a(v))$, then $a(u)=a(v)$.
    \end{claim}

    \begin{proofofclaim}
        By assumption on $u$ and $v$,
        \begin{equation}
            \label{claimAncSepsEq1}
            \dist_G(v,a(v))+\sigma(a(v))\ell=\dist_G(u,v)+\dist_G(u,a(v))+\sigma(a(v))\ell.
        \end{equation}
        By triangle inequality,
        \begin{equation}
            \label{claimAncSepsEq2}
            \dist_G(v,a(u))+\sigma(a(u))\ell\leq \dist_G(u,v)+\dist_G(u,a(u))+\sigma(a(u))\ell.
        \end{equation} 
        By choice of $a(u)$ and $a(v)$,
        \begin{subequations}
            \begin{align}
                \label{claimAncSepsEq3a}
                \dist_G(u,a(u))+\sigma(a(u))\ell&\leq \dist_G(u,a(v))+\sigma(a(v))\ell\text{, and}\\
                \label{claimAncSepsEq3b}
                \dist_G(v,a(v))+\sigma(a(v))\ell&\leq \dist_G(v,a(u))+\sigma(a(u))\ell.
            \end{align}
        \end{subequations}
        Applying \eqref{claimAncSepsEq2}, \eqref{claimAncSepsEq3a}, \eqref{claimAncSepsEq1} (in that order) and then comparing with \eqref{claimAncSepsEq3b},
        \begin{equation}
            \label{claimAncSepsEq4}
            \dist_G(v,a(u))+\sigma(a(u))\ell= \dist_G(v,a(v))+\sigma(a(v))\ell.
        \end{equation}
        By choice of $a(v)$, \eqref{claimAncSepsEq4} implies that $a(u)$ does not have higher priority than $a(v)$ in $Q$.

        Applying \eqref{claimAncSepsEq1}, \eqref{claimAncSepsEq4}, and \eqref{claimAncSepsEq2} (in that order), cancelling out $\dist_G(u,v)$, and then comparing with \eqref{claimAncSepsEq3a},
        \begin{equation}
            \label{claimAncSepsEq5}
            \dist_G(u,a(u))+\sigma(a(u))\ell= \dist_G(u,a(v))+\sigma(a(v))\ell.
        \end{equation}
        By choice of $a(u)$, \eqref{claimAncSepsEq5} implies that $a(v)$ does not have higher priority than $a(u)$ in $Q$. Since for all pairs of distinct vertices, one will have higher priority in $Q$ than the other, this implies $a(u)=a(v)$, as desired.
    \end{proofofclaim}

    We next show that if a vertex is in the security of $B$, it lies in the security of the centre $a(v)$ it belongs to.
    
    \begin{claim}
        \label{claimSecureInAnc}
        For each $v\in S(B)\cap J$, $v\in S(a(v))$.
    \end{claim}

    \begin{proofofclaim}
        Since $v\in S(B)$, there exists $x\in B$ such that $v\in S(x)$, and thus $\dist_G(v,x)\leq (|B|-1-\sigma(x))\ell$. Hence, $\dist_G(v,x)+\sigma(x)\ell\leq (|B|-1)\ell$. Since $v\in J$, $a(v)$ is well-defined, and by definition we have $\dist_G(v,a(v))+\sigma(a(v))\ell\leq \dist_G(v,x)+\sigma(x)\ell\leq (|B|-1)\ell$. Thus, $\dist_G(v,a(v))\leq (|B|-1-\sigma(a(v)))\ell$, so $v\in S(a(v))$, as desired.
    \end{proofofclaim}
    
    For each $x\in B$, let $d(x)$ be the set consisting of the $v\in J$ such that $a(v)=x$ and $v\in S(x)$. Let $D:=\bigcup_{x\in B}d(x)$. Note that $D\subseteq J$.

    \begin{claim}
        \label{claimDIsSecurity}
        $D=S(B)\cap J$.
    \end{claim}

    \begin{proofofclaim}
        For each $v\in D$, we have $v\in S(a(v))\subseteq S(B)$. Since $D\subseteq J$, it follows that $v\in S(B)\cap J$. For each $v\in S(B)\cap J$, by \cref{claimSecureInAnc}, $v\in S(a(v))$. So $v\in d(a(v))$, and thus $v\in D$. The claim follows.
    \end{proofofclaim}

    If $D=J$, then by \cref{claimDIsSecurity}, $J\subseteq S(B)$. Hence, in this case, the lemma is true with $B':=B$ (as $S'(B')=S(B)$). Thus and since $D\subseteq J$, we may assume that $D\subsetneq J$.

    Let $A:=\{x\in B:d(x)\neq \emptyset\}$. Noting that the $d(x)$, $x\in B$, are pairwise disjoint and disjoint to $J\setminus D$, observe that $|A|+|J\setminus D|\leq |J|$.

    Let $m:=\min(|J\setminus D|,|B\setminus A|)$, and let $X$ be a subset of $B\setminus A$ of size exactly $m$ (which exists by choice of $m$). Let $B':=(B\setminus X)\cup (J\setminus D)$.

    By \cref{claimDIsSecurity} and since $B\subseteq S(B)$, observe that $J\cap B\subseteq J\cap D$. Thus, $J\setminus D$ is disjoint to $B$. Hence and since $X\subseteq B$ and has size $m$, $|B'|=|B|-m+|J\setminus D|$. If $m=|J\setminus D|$, then observe that $|B'|=|B|$. Otherwise, observe that $m=|B\setminus A|\leq |J\setminus D|$, and $|B'|=|B|-|B\setminus A|+|J\setminus D|$. As $|B\setminus A|\leq |J\setminus D|$, we obtain $|B|\leq |B'|$. However, we also have that $|B|-|B\setminus A|+|J\setminus D|=|A|+|J\setminus D|\leq |J|$. So $|B'|\leq |J|$. Hence, in either scenario, we find that $|B|\leq |B'|\leq \max(|B|,|J|)$.

    Recalling that $J\setminus B$ is disjoint to $B$, we have $B\cap B'=B\setminus X$ and $B'\setminus B=J\setminus D=J\setminus S(B)$ by \cref{claimDIsSecurity}. Note also that $B\setminus B'=X$.

    We next show that $\sigma'$ is smaller than $\sigma$ on $B\cap B'$, and give a sufficient condition for it to be strictly smaller.

    \begin{claim}
        \label{claimPenaltyDecrease}
        For each $x\in B\cap B'$, $\sigma'(x)\leq \sigma(x)$. Further, if there exists $y\in B\setminus B'$ with higher priority in $Q$ than $x$, then $\sigma'(x)\leq \sigma(x)-1$.
    \end{claim}

    \begin{proofofclaim}
        Recall that $\sigma(x)$ and $\sigma'(x)$ count the number of vertices in $B$ and $B'$ respectively with higher priority than $x$ in $Q$. Since $x\in B\cap B'\subseteq B$ has higher priority in $Q$ than each $v\in B'\setminus B=J\setminus S(B)$, it follows that the only vertices in $B'$ with higher priority than $x$ are contained in $B$. This gives $\sigma'(x)\leq \sigma(x)$. Further, if such a vertex $y$ exists, then there is at least one fewer vertex in $B'$ with higher priority in $Q$ than $x$ compared to $B$. This gives $\sigma'(x)-1\leq \sigma(x)$.
    \end{proofofclaim}

    We now show that $D$ is contained in the `new' security of $B\cap B'$.

    \begin{claim}
        \label{claimDSecureBySubset}
        $D\subseteq S'(B\cap B')$.
    \end{claim}

    \begin{proofofclaim}
        Fix $v\in D$. Since $v\in D$, $v\in d(x)$ for some $x\in B$. By definition of $d(x)$, we find that $x=a(v)$ and $v\in S(x)=S(a(v))$. As $v\in d(a(v))$, we have that $a(v)\in A$, and thus $a(v)\in B\setminus X\subseteq B'$ (as $X\subseteq B\setminus A$ and since $A\subseteq B$). So $a(v)\in B\cap B'$.

        Since $v\in S(a(v))$, we have $\dist_G(v,a(v))\leq (|B|-\sigma(a(v)))\ell$. By \cref{claimPenaltyDecrease}, $\sigma'(a(v))\leq \sigma(a(v))$. Recall also that $|B|\leq |B'|$. Thus, $\dist_G(v,a(v))\leq (|B'|-\sigma'(a(v)))\ell$. This gives $v\in S'(a(v))\subseteq S'(B\cap B')$, as desired.
    \end{proofofclaim}

    Noting that $J\setminus D\subseteq S'(J\setminus D)=S'(B'\setminus B)$ and observing that $S'(B')=S'(B'\setminus B)\cup S'(B'\cap B)$, it follows from \cref{claimDSecureBySubset} that $J\subseteq S'(B')$.

    Fix $x\in B\setminus B'=X$. We must show that $L(x)\subseteq S'(B\cap B')$. Let $v\in L(x)$. So:
    \begin{enumerate}
        \item there exists $z\in J$ such that $\dist_G(v,z)\leq (|B|-\sigma(x))\ell - \dist_G(x,z)$, and
        \item there does not exist $x'\in B$ with higher priority than $x$ in $Q$ such that $v\in C(x')$.
    \end{enumerate}

    Since $J\cap S(B)$ separates $B$ from $J$, there exists $w\in J\cap S(B)$ such that $\dist_G(z,a(z))=\dist_G(z,w)+\dist_G(w,a(z))$. By \cref{claimAncSeps} (with $(u,v):=(w,z)$), $a(z)=a(w)$. By \cref{claimSecureInAnc}, $w\in S(a(w))$. Hence, $w\in d(a(w))$, and thus $a(w)=a(z)\in A\subseteq B\setminus X=B\cap B'$.

    Since $a(z)\in A$ and $x\notin A$ (as $x\in B\setminus B'=X$), $a(z)\neq x$. By definition of $a(z)$, we have that $\dist_G(z,a(z))+\sigma(a(z))\ell\leq \dist_G(z,x)+\sigma(x)\ell$. Therefore, $(|B|-\sigma(x))\ell - \dist_G(x,z)\leq (|B|-\sigma(a(z)))\ell - \dist_G(a(z),z)$. Thus and since $v\in L(x)$, we obtain $\dist_G(v,z)\leq (|B|-\sigma(a(z)))\ell - \dist_G(a(z),z)$. Hence and by triangle inequality, we obtain $\dist_G(v,a(z))\leq \dist_G(v,z)+\dist_G(a(z),z)\leq (|B|-\sigma(a(z)))\ell$. Thus, $v\in C(a(z))$. Since $v\in L(x)$ (and since $x\neq a(z)$), this implies that $x$ has higher priority than $a(z)$ in $Q$.

    By \cref{claimPenaltyDecrease} (with $(x,y):=(a(z),x)$), $\sigma'(a(z))\leq \sigma(a(z))-1$. Since $\dist_G(v,a(z))\leq (|B|-\sigma(a(z)))\ell\leq (|B'|-\sigma(a(z)))\ell$ (as $|B|\leq |B'|$), this gives $\dist_G(v,a(z))\leq (|B'|-1-\sigma'(a(z)))\ell$. So $v\in S'(a(z))\subseteq S'(B\cap B')$. It follows that $L(x)\subseteq S'(B\cap B')$.
    
    This completes the proof of the lemma.
\end{proof}

We can now get into the main proof.

\begin{theorem}
    \label{compressingExists}
    Let $k\in \ds{N}$ and $\ell\in \RanIn{R}{0}$, and let $(J_t:t\in V(T))$ be a tree-decomposition of a graph $G$ of width at most $k$. Then there exists a graph $H$, a $T$-decomposition $(B_t:t\in V(T))$ of $H$ of width at most $k$, and an $\ell$-compressing $H$-indexed partition $(P_x:x\in V(H))$ of $G$ whose parts each have weak diameter in $G$ at most $2(k+1)\ell$.
\end{theorem}

\begin{proof}
    Fix a root $r$ of $T$. For each $v\in V(G)$, let $T_v$ be the subtree of $T$ induced by the vertices $t\in V(T)$ such that $v\in J_t$, and let $r(v)$ be the (unique) vertex of $T_v$ closest to $r$. We call $r(v)$ the root of $v$. Note that $r(v)$ is an ancestor of each $t\in V(T_v)$ (each $t\in V(T)$ with $v\in J_t$).

    Observe that we can find a total ordering $\preceq$ of $V(G)$ such that for distinct $u,v\in V(G)$, if $r(u)$ is a strict ancestor of $r(v)$, then $v\prec u$. So $Q:=(G,\preceq)$ is a graph-priority pair. By definition of $\preceq$, when distinct $u,v\in V(G)$ are such that $r(u)$ is a strict ancestor of $r(v)$, $u$ has priority in $Q$ than $v$.

    For each $t\in V(T)$, we will define a set $B_t\subseteq V(G)$. For brevity, we define $\sigma_t:=\sigma_Q^{B_t}$, which we call the penalty map at $t$. For $B^*\subseteq B_t$, we also define $(S_t(B^*),C_t(B^*)):=(S_Q^{B_t,\ell}(B^*),C_Q^{B_t,\ell}(B^*))$, which we call the security and coverage respectively of $B^*$ at $t$. If $B^*=\{x\}$ for some $x\in B$, we abbreviate $S_t(\{x\})$ as $S_t(x)$ and $C_t(\{x\})$ as $C_t(x)$. Finally, for $x\in B_t$ and child $t^*$ of $t$, we define $L_t(x,t^*):=L_Q^{B_t,\ell}(x,J_{t^*})$, which we call the leak from $x$ through $t^*$ at $t$.
    
    For each $t\in V(T)$, $B_t$ will satisfy the following properties: 
    \begin{enumerate}
        \item $J_t\subseteq S_t(B_t)$, and
        \item $r(x)$ is an ancestor of $t$ for each $x\in B_t$.
    \end{enumerate}

    If $t\neq r$, let $t'$ be the parent of $t$. We then further require that:

    \begin{enumerate}
        \item $|B_t'|\leq |B_t|\leq \max(|B_t'|,|J_t|)$,
        \item $B_t\setminus B_{t'}=J_t\setminus S_{t'}(B_{t'})$, and
        \item for each $x\in B_{t'}\setminus B_t$, $L_{t'}(x,t)\subseteq S_t(B_t\cap B_{t'})$.
    \end{enumerate}

    We define $B_t$ inductively on the distance from $t$ to $r$. In the base case $t=r$, we set $B_r:=J_r$, which trivially satisfies the above requirements (as $J_r\subseteq S_r(J_r)$).

    We now proceed to the inductive step. Note that $t\neq r$, so the parent $t'$ of $t$ is defined. Further, $B_{t'}$ is already defined, and has the above properties.

    Since $J_{t'}\subseteq S_{t'}(B_{t'})$, observe that $J_t\cap J_{t'}\subseteq J_t\cap S_{t'}(B_{t'})$ and $J_t\setminus S_{t'}(B_{t'})\subseteq J_t\setminus J_{t'}$. Since $B_{t'}\subseteq V(G)$ and since $r(x)$ is an ancestor of $t'$ for each $x\in B_{t'}$, observe that $J_t\cap J_{t'}$, and thus $J_t\cap S_{t'}(B_{t'})$, separates $B_{t'}$ from $J_t$ in $G$. For each $u\in B_{t'}$ and $v\in J_t\setminus S_{t'}(B_{t'})\subseteq J_t\setminus J_{t'}$, observe that $r(v)=t$, and that $r(u)$ is an ancestor of $t'$. Hence, $r(u)$ is a strict ancestor of $r(v)$, so $u$ has higher priority in $Q$ than $v$.
    
    Thus, we can apply \cref{changeProviders} with $(B,J):=(B_{t'},J_t)$. Let $B_t:=B'$ be the result. We have that:
    \begin{enumerate}
        \item $|B_t'|\leq |B_t|\leq \max(|B_t'|,|J_t|)$,
        \item $J_t\subseteq S_t(B_t)$,
        \item $B_t\setminus B_{t'}=J_t\setminus S_{t'}(B_{t'})$, and
        \item for each $x\in B_{t'}\setminus B_t$, $L_{t'}(x,t)\subseteq S_t(B_t\cap B_{t'})$.
    \end{enumerate}

    To satisfy the inductive conditions, it remains only to show that for each $x\in B_t$, $r(x)$ is an ancestor of $t$. If $x\in B_{t'}$, then $r(x)$ is an ancestor of $t'$ and thus $t$. Otherwise, since $B_t\setminus B_{t'}=J_t\setminus S_{t'}(B_{t'})\subseteq J_t\setminus J_{t'}$, we have $r(x)=t$. So $r(x)$ is trivially an ancestor of $t$.

    Let $H$ be the graph with vertex set $\bigcup_{t\in V(T)}B_t$ and edges between distinct vertices $u,v\in V(H)$ if and only if there exists $t\in V(T)$ such that $u,v\in B_t$. Note that $V(H)\subseteq V(G)$.

    For each $v\in V(G)$, let $a(v)$ be the vertex $x\in V(H)$ of maximum priority in $Q$ such that: 
    \begin{enumerate}
        \item $r(x)$ is an ancestor of $r(v)$, and
        \item there exists an ancestor $t^*(v):=t^*$ of $r(v)$ such that $x\in B_{t^*}$ and $v\in C_{t^*}(x)$.
    \end{enumerate}
    
    We remark that such a vertex $a(v)$ exists, as $v\in J_{r(v)}\subseteq S_{r(v)}(B_{r(v)})\subseteq C_{r(v)}(B_{r(v)})$, so there exists $y\in B_{r(v)}$ with $v\in C_{r(v)}(y)$. Hence, we could take $(x,t^*):=(y,r(v))$, since $r(y)$ is an ancestor of $r(v)$ as $y\in B_{r(v)}$.
    
    For each $x\in V(H)$, let $P_x:=\{v\in V(G):a(v)=x\}$. Observe that $(P_x:x\in V(H))$ is an $H$-indexed partition of $G$. We show that $(H,(B_t:t\in V(T)),(P_x:x\in V(H)))$ satisfy the theorem.

    We first need to control the size of the `bags' $B_t$, $t\in V(T)$.

    \begin{claim}
        \label{claimBagSizes}
        For each $t\in V(T)$ and each ancestor $t^*$ of $t$, $|B_{t^*}|\leq |B_t|\leq k+1$.
    \end{claim}

    \begin{proofofclaim}
        By induction on $\dist_T(r,t)+\dist_T(r,t^*)$. In the base case, $r=t^*=t$. Thus, $B_{t^*}=B_t=B_r=J_r$, so $|B_{t^*}|=|B_t|\leq k+1$ as $(J_t:t\in V(T))$ has width at most $k$. So we may proceed to the inductive step. In particular, observe that $t\neq r$, so the parent $t'$ of $t$ is well-defined.

        If $t=t^*$, then observe that $\dist_T(r,t)+\dist_T(r,r)<\dist_T(r,t)+\dist_T(r,t^*)$. So we may apply induction with $(t,t^*):=(t,r)$. We find that $|B_r|\leq |B_t|\leq k+1$. Thus, $|B_{t^*}|=|B_t|\leq k+1$, as desired.
        
        So we may assume that $t^*$ is a strict ancestor of $t$. Thus, $t^*$ is an ancestor of $t'$. Observe that $\dist_T(r,t')+\dist_T(r,t^*)<\dist_T(r,t)+\dist_T(r,t^*)$. So we may apply induction with $(t,t^*):=(t',t^*)$. We find that $|B_{t^*}|\leq |B_{t'}|\leq k+1$. Since $t'$ is the parent of $t$, we also have that $|B_{t'}|\leq |B_t|\leq \max(|B_{t'}|,|J_t|)$. Since $(J_t:t\in V(T))$ has width at most $k$, this gives $|B_{t^*}|\leq |B_t|\leq k+1$, as desired.
    \end{proofofclaim}

    In particular, for each $t\in V(T)$, by applying \cref{claimBagSizes} with $t^*:=r$, we find that $|B_t|\leq k+1$.

    We can now show that the parts of $(P_x:x\in V(H))$ have bounded weak diameter in $G$.

    \begin{claim}
        \label{claimWeakRadius}
        For each $x\in V(H)$ and $v\in P_x$, $\dist_G(x,v)\leq (k+1)\ell$.
    \end{claim}

    \begin{proofofclaim}
        Since $v\in P_x$, $a(v)=x$. Thus, by definition of $a(v)=x$ and $t^*(v)$, $v\in C_{t^*(v)}(x)$. Hence, $\dist_G(x,v)\leq (|B_{t^*(v)}|-\sigma_{t^*(v)}(x))\ell$. Since $\sigma_{t^*(v)}(x)\geq 0$ and since $|B_{t^*(v)}|\leq k+1$ (by \cref{claimBagSizes}), this gives $\dist_G(x,v)\leq (k+1)\ell$, as desired.
    \end{proofofclaim}

    It follows from \cref{claimWeakRadius} that for each $x\in V(H)$, $P_x$ has weak diameter in $G$ at most $2(k+1)\ell$.

    We next show that $(B_t:t\in V(T))$ is the desired $T$-decomposition of $H$. For each $x\in V(H)$, let $T_x'$ be the subgraph of $T$ induced by the vertices $t\in V(T)$ such that $x\in B_t$. The main difficulty is showing that $T_x'$ is connected. For this, we show that the entire path from $t$ to $r(x)$ (including $r(x)$) is contained in $V(T_x')$.

    \begin{claim}
        \label{claimIsSubtree}
        Let $x\in V(H)$ and $t^*\in V(T_x')$. Then for each vertex $t'$ on the path from $t^*$ to $r(x)$ in $T$, we have $x\in B_{t'}$.
    \end{claim}

    \begin{proofofclaim}
        By induction on the distance from $t'$ to $t^*$. In the base case, $t'=t^*$, and the claim is trivial (as $x\in B_{t^*}$ since $t^*\in V(T_x')$). So we may proceed to the inductive step. In particular, $t'\neq t^*$.

        Since $t^*\in V(T_x')$, we have $x\in B_{t^*}$. Hence, recall that $r(x)$ is an ancestor of $t^*$. Thus and since $t'\neq t^*$, there is a child $t$ of $t'$ on the path from $t^*$ to $r(x)$. Observe that $t$ is closer to $t^*$ than $t'$ is, and thus by induction $x\in B_t$.

        Presume, for a contradiction, that $x\notin B_{t'}$. So $x\in B_t\setminus B_{t'}=J_t\setminus S_{t'}(B_{t'})$. Since $J_{t'}\subseteq S_{t'}(B_{t'})$, this gives $x\in J_t\setminus J_{t'}$. Hence, we have $t,r(x)\in V(T_x)$, but $t'\notin V(T_x)$. This contradicts the fact that $T_x$ is a subtree of $T$, as $t'$ is on the unique path from $t$ to $r(x)$ in $t$ (since $r(x)$ is an ancestor of $t^*$ and since $t$ is a child of $t'$). Hence, we have $x\in B_{t'}$. The claim follows.
    \end{proofofclaim}

    We can now show that $(B_t:t\in V(T))$ is the desired $T$-decomposition of $H$. By definition of $H$, $B_t\subseteq V(H)$ for each $t\in V(T)$. For each $x\in V(H)$, by definition of $H$, $T_x'$ is nonempty, and by \cref{claimIsSubtree}, $T_x'$ is connected (as the path from any $t^*\in V(T_x')$ to $r(x)$ is contained in $T_x'$). For each $uv\in E(H)$, by definition of $H$, there exists $t\in V(T)$ such that $u,v\in B_t$. Therefore, $(B_t:t\in V(T))$ is a $T$-decomposition of $H$. Recalling that $|B_t|\leq k+1$ for each $t\in V(T)$ (by \cref{claimBagSizes}), we find that $(B_t:t\in V(T))$ has width at most $k$, as desired.

    For each $x\in V(H)$, let $r'(x)$ be the vertex of $T_x'$ closest to $r$ (which is well-defined and uniquely defined, as $T_x'$ is a nonempty subtree). Note that $r'(v)$ is an ancestor of each $t\in V(T_v')$ (each $t\in V(T)$ with $x\in B_t$).

    \begin{claim}
        \label{claimSameRoot}
        For each $x\in V(H)$, $r(x)=r'(x)$.
    \end{claim}

    \begin{proofofclaim}
        By \cref{claimIsSubtree} (and since $T_x'$ is nonempty), $r(x)\in V(T_x')$. So $r(x)$ is a descendant of $r'(x)$. As $x\in B_{r'(x)}$, recall that $r(x)$ is an ancestor of $r'(x)$. Thus, $r(x)=r'(x)$.
    \end{proofofclaim}

    We are now left with the hardest step, showing that $(P_x:x\in V(H))$ is $\ell$-compressing in $G$ (as an $H$-indexed partition).

    We first need the following fact regarding, for $v\in V(G)$, the root of $a(v)$.

    \begin{claim}
        \label{claimCoverageAnc}
        Let $v\in V(G)$ and $t,t^*\in V(T)$ be such that $t^*$ is an ancestor of $r(v)$ and a descendant of $t$. If $v\in C_{t^*}(B_t\cap B_{t^*})$, then $r(a(v))$ is an ancestor of $t$.
    \end{claim}

    \begin{proofofclaim}
        So there exists $x\in B_t\cap B_{t^*}$ such that $v\in C_{t^*}(x)$. As $x\in B_t$, $r(x)$ is an ancestor of $t$, and thus $t^*$ and $r(v)$. So we have that:
        \begin{enumerate}
            \item $r(x)$ is an ancestor of $r(v)$, and
            \item there exists an ancestor $t^*$ of $r(v)$ such that $x\in B_{t^*}$ and $v\in C_{t^*}(x)$.
        \end{enumerate}
        Hence, by choice of $a(v)$, $a(v)$ has higher priority than $x$ in $Q$.
        
        Recall that $r(a(v))$ and $r(x)$ are ancestors of $r(v)$. Thus, $r(a(v))$ is related to $r(x)$. Since $a(v)$ has higher priority than $x$ in $Q$, $r(x)$ cannot be a strict ancestor of $r(a(v))$. Thus, $r(a(v))$ is an ancestor of $r(x)$. Recalling that $r(x)$ is an ancestor of $t$, this implies that $r(a(v))$ is an ancestor of $t$, as desired.
    \end{proofofclaim}

    We can now relate the roots of $a(u),a(v)$ whenever $u,v\in V(G)$ are at distance at most $\ell$.

    \begin{claim}
        \label{claimRelatedRoots}
        If $u,v\in V(G)$ are such that $\dist_G(u,v)\leq \ell$, then $r(a(u))$ and $r(a(v))$ are related, and both are ancestors of both $r(u)$ and $r(v)$.
    \end{claim}

    \begin{proofofclaim}
        Let $t\in V(T)$ be the (unique) vertex farthest from $r$ that is an ancestor of both $r(u)$ and $r(v)$. Note that $J_t$ separates $J_{r(u)}$ from $J_{r(v)}$. Thus and since $u\in J_{r(u)}$ and $v\in J_{r(v)}$, observe that there exists $z\in J_t$ such that $\dist_G(u,v)=\dist_G(u,z)+\dist_G(z,v)$. In particular, observe that $\dist_G(u,z),\dist_G(v,z)\leq \ell$. Since $z\in J_t\subseteq S_t(B_t)$, we obtain $u,v\in N_G^{\ell}(S_t(B_t))\subseteq C_t(B_t)$. By \cref{claimCoverageAnc} (with $t^*:=t$), $r(a(u)),r(a(v))$ are both ancestors of $t$, and thus both $r(u)$ and $r(v)$. This also implies that $r(a(u))$ and $r(a(v))$ are related, completing the proof of the claim.
    \end{proofofclaim}
    
    We also need the following fact about the penalty of a vertex that appears in multiple bags $B_t$.

    \begin{claim}
        \label{claimConsistentPenalty}
        For each $t\in V(T)$, each ancestor $t^*$ of $t$, and each $x\in B_t\cap B_{t^*}$, we have $\sigma_t(x)\leq \sigma_{t^*}(x)$.
    \end{claim}

    \begin{proofofclaim}
        By \cref{claimIsSubtree,claimSameRoot}, observe that for each $y\in B_t\setminus B_{t^*}$, $r'(y)=r(y)$ is a strict descendant of $t^*$, and thus $r'(x)=r(x)$ (as $x\in B_{t^*}$). Therefore, $y$ has lower priority than $x$ in $Q$.
    
        Recall that $\sigma_t(x)$ and $\sigma_{t^*}(x)$ count the number of vertices in $B_t$ and $B_{t^*}$ respectively that have higher priority than $x$ in $Q$. It follows then that $\sigma_t(x)\leq \sigma_{t^*}(x)$, as desired.
    \end{proofofclaim}

    Finally, we need the following result regarding, for $v\in V(G)$, the penalty of vertices lying in a bag $B_t$ containing $a(v)$.

    \begin{claim}
        \label{claimNoLowerPenalty}
        If $v\in V(G)$ and $t\in V(T)$ are such that $a(v)\in B_t$ and $t$ is an ancestor of $r(v)$, then for each $x\in B_t$ with $v\in C_t(x)$, $x$ does not have higher priority than $a(v)$ in $Q$.
    \end{claim}

    \begin{proofofclaim}
        As $x\in B_t$, $r(x)$ is an ancestor of $t$, and thus $r(v)$. So we have that:
        \begin{enumerate}
            \item $r(x)$ is an ancestor of $r(v)$, and
            \item there exists an ancestor $t^*:=t$ of $r(v)$ such that $x\in B_{t^*}$ and $v\in C_{t^*}(x)$.
        \end{enumerate}
        Hence, by definition of $a(v)$, $x$ does not have higher priority than $a(v)$ in $Q$.
    \end{proofofclaim}

    To prove that $(P_x:x\in V(H))$ is $\ell$-compressing, for $u,v\in V(G)$ with $\dist_G(u,v)\leq \ell$, we need to find a bag $B_t$ with $a(u),a(v)\in B_t$ (as then $a(u),a(v)$ are adjacent in $H$).

    \begin{claim}
        \label{claimCentresOverlap}
        If $u,v\in V(G)$ are such that $\dist_G(u,v)\leq \ell$, then there exists $w\in \{u,v\}$ such that $a(u),a(v)\in B_{r(a(w))}$.
    \end{claim}

    \begin{proofofclaim}
        By \cref{claimRelatedRoots}, $r(a(u))$ and $r(a(v))$ are related and are both ancestors of both $r(u)$ and $r(v)$. Without loss of generality, say that $r(a(v))$ is an ancestor of $r(a(u))$.

        Presume, for a contradiction, that $a(v)\notin B_{r(a(u))}$. Note that $a(v)\in B_{r'(a(v))}=B_{r(a(v))}$ by \cref{claimSameRoot}. Thus and since $r(a(v))$ is an ancestor of $r(a(u))$), there exists a vertex $t'\in V(T)$ and a child $t$ of $t'$ such that:
        \begin{enumerate}
            \item $t'$ is a descendant of $r(a(v))$,
            \item $t$ is an ancestor of $r(a(u))$,
            \item $a(v)\in B_{t'}$, and
            \item $a(v)\notin B_t$.
        \end{enumerate}
        
        Note that $t'$ is a strict ancestor of $r(a(u))$, and that $t$ is a strict descendant of $r(a(v))$. Further, since $r(a(u))$ is an ancestor of $r(v)$, both $t$ and $t'$ are on the path from $r(v)$ to $r(a(v))$ in $T$.

        Recall that $v\in C_{t^*(v)}(a(v))$, so $\dist_G(v,a(v))\leq (|B_{t^*(v)}|-\sigma_{t^*(v)}(a(v)))\ell$. Since $t$ is on the path from $r(v)$ to $r(a(v))$ in $T$, $J_t$ separates $J_{r(v)}$ from $J_{r(a(v))}$ in $G$. Thus and since $v\in J_{r(v)}$, $a(v)\in J_{r(a(v))}$, there exists $z\in J_t$ such that $\dist_G(a(v),v)=\dist_G(a(v),z)+\dist_G(v,z)$. Hence, we have that $\dist_G(v,z)\leq (|B_{t^*(v)}|-\sigma_{t^*(v)}(a(v)))\ell-\dist_G(a(v),z)$.
        
        Recall that $t^*(v)$ is an ancestor of $r(v)$. Thus, $t^*(v)$ is related to both $t$ and $t'$. Recall also that $a(v)\in B_{t^*(v)}$. Thus, by \cref{claimIsSubtree,claimSameRoot} and since $a(v)\in B_{t'}\setminus B_t$, $t^*(v)$ is a strict ancestor of $t$ and an ancestor of $t'$. Hence, by \cref{claimBagSizes,claimConsistentPenalty} (with $t:=t'$), we have $|B_{t^*(v)}|\leq |B_{t'}|$ and $\sigma_{t'}(a(v))\leq \sigma_{t^*(v)}(a(v))$. Therefore, we obtain $\dist_G(v,z)\leq (|B_{t'}|-\sigma_{t'}(a(v)))\ell-\dist_G(a(v),z)$. By \cref{claimNoLowerPenalty} (with $t:=t'$), there does not exist $x\in B_{t'}$ with higher priority than $a(v)$ in $Q$ such that $v\in C_{t'}(x)$. Thus (and since $z\in J_t$), $v\in L_{t'}(a(v),t)$.

        Since $t'$ is the parent of $t$ and since $a(v)\in B_{t'}\setminus B_t$, we know that $L_{t'}(a(v),t)\subseteq S_t(B_t\cap B_{t'})$. So $v\in S_t(B_t\cap B_{t'})$, and thus $u\in N_G^{\ell}(v)\subseteq N_G^{\ell}(S_t(B_t\cap B_{t'}))\subseteq C_t(B_t\cap B_{t'})$. By \cref{claimCoverageAnc} (with $(v,t,t^*):=(u,t',t)$, recalling that $t$ is an ancestor of $r(a(u))$ and thus $r(u)$), $r(a(u))$ is an ancestor of $t'$. But this contradicts the fact that $t'$ is a strict ancestor of $r(a(u))$. 
        
        Thus, we can conclude that $a(v)\in B_{r(a(u))}$. Noting that $a(u)\in B_{r'(a(u))}=B_{r(a(u))}$ by \cref{claimSameRoot}, the claim is satisfied with $w:=u$. (Recall that $u$ was selected so that $r(a(u))$ was a descendant of $r(a(v))$).
    \end{proofofclaim}

    Let distinct $x,y\in V(H)$ be such that $\dist_G(P_x,P_y)\leq \ell$. So there exists $u\in P_x$, $v\in P_y$ with $\dist_G(u,v)\leq \ell$. Note that $a(u)=x$ and $a(v)=y$. By \cref{claimCentresOverlap}, there exists $w\in \{u,v\}$ such that $a(u)=x,a(v)=y\in B_{r(a(w))}$. Thus, by definition of $H$, $x,y$ are adjacent in $H$. Hence, $(P_x:x\in V(H))$ is $\ell$-compressing. This completes the proof of the theorem.
\end{proof}

\cref{compressingFuncExists} follows directly from \cref{compressingExists}.

\paragraph{Acknowledgements.} The author thanks David Wood for his supervision and suggestions to improve this paper.
      
\bibliography{Ref.bib}
\end{document}